\documentclass[a4paper,11pt]{article}
\usepackage{amsmath,amssymb,amsthm,graphicx}
\usepackage{fullpage}
\usepackage{enumitem}
\usepackage{multirow}
\usepackage{makecell}
\usepackage{arydshln}

\usepackage[utf8]{inputenc}
\usepackage{algorithm}
\usepackage[noend]{algpseudocode}

\newtheorem{theorem}{Theorem}[section]
\newtheorem{proposition}[theorem]{Proposition}
\newtheorem{lemma}[theorem]{Lemma}
\newtheorem{corollary}[theorem]{Corollary}

\theoremstyle{definition}
\newtheorem{example}[theorem]{Example}
\newtheorem{definition}[theorem]{Definition}

\newcommand{\G}{\ensuremath{\mathcal{G}}}

\newcommand{\A}{\ensuremath{\mathcal{A}}}

\title{Strong external difference families in abelian and non-abelian groups}
\author{Sophie Huczynska \thanks{School of Mathematics \& Statistics, University of St Andrews, St Andrews, UK; email: sh70@st-andrews.ac.uk}, Christopher Jefferson \thanks{School of Computer Science, University of St Andrews, St Andrews, UK; email: caj21@st-andrews.ac.uk} and Silvia Nep\v{s}insk\'{a} \thanks{School of Computer Science, University of St Andrews, St Andrews, UK}}
\date{}
\begin{document}
\maketitle

\begin{abstract}
Strong external difference families (SEDFs) have applications to cryptography and are rich combinatorial structures in their own right.  We extend the definition of SEDF from abelian groups to all finite groups, and introduce the concept of equivalence for external difference families (EDFs) and SEDFs.  We prove new recursive results which significantly extend known constructions for SEDFs and generalized SEDFs (GSEDFs) in cyclic groups, and present the first family of non-abelian SEDFs.  We establish the existence of at least two non-equivalent $(k^2+1, 2, k,1)$-SEDFs for every $k>2$, and we begin the task of enumerating SEDFs, via a computational approach which yields complete results for all groups up to order $24$.
\end{abstract}
\begin{center}

\end{center}
\renewcommand{\thefootnote}{\fnsymbol{footnote}} 
\renewcommand{\thefootnote}{\arabic{footnote}} 

\section{Introduction}

There has been considerable recent interest in strong external difference families (SEDFs), which have applications to cryptography and are rich combinatorial structures in their own right (see \cite{BaJiWeZh}, \cite{HuPa}, \cite{JeLi}, \cite{MaSt},\cite{PaSt}, \cite{WeYaFuFe}).  Most of this activity has centred around identifying parameters for which a SEDF exists or does not exist, and obtaining constructions via the application of classic techniques such as cyclotomy.  Up till now, all SEDFs have been in abelian groups.

We are motivated by the following questions: what is the situation for SEDFs in general finite groups, not simply abelian groups?  For given parameters, which groups contain SEDFs with these parameters?  How many ``different'' SEDFs exist with the same parameters?  

In this paper, we introduce the notion of equivalence for SEDFs, and characterise the order of groups with admissible parameters for SEDFs.  We present a recursive framework for constructing families of SEDFs with $\lambda=1$ (and related generalized SEDFs) in cyclic groups,  which encompasses known results on SEDFs and GSEDFs.  We present the first non-abelian SEDFs: a construction for an infinite non-abelian family using dihedral groups.  We establish the existence of at least two non-equivalent $(k^2+1, 2, k,1)$-SEDFs for every $k>2$.  Finally, we begin the task of enumerating SEDFs, and present complete results for all groups up to order $24$, underpinned by a computational approach using constraint satisfaction programming. 

\section{Strong External Difference Families}

External difference families were first introduced in \cite{OgKuStSi} in relation to AMD codes, while strong EDFs were subsequently introduced in \cite{PaSt} to correspond to strong AMD codes.  They were defined in abelian groups. 

In this paper, we extend the concept of EDF and SEDF in the natural way to any group of order $n$, whether abelian or non-abelian.  The definitions precisely correspond to the originals, with the removal of the word ``abelian''.  Since the differences are defined in terms of ordered pairs, there is no ambiguity in this definition.  For abelian groups, additive notation is generally used; however when we focus on the non-abelian and general cases, we will adopt multiplicative notation.

\begin{definition}[External difference family]
Let $\G$ be a group of order $n$.
An {\bf $(n,m,k,\lambda)$-external difference family} 
(or {\bf $(n,m,k,\lambda)$-EDF})
is a set of $m \geq 2$ disjoint $k$-subsets of $\G$, say $A_1, \dots , A_m$, such that the multiset
\begin{equation*}
M=\{ xy^{-1}: x \in A_i, y \in A_j, i \neq j \}
\end{equation*}
comprises $\lambda$ occurrences of each non-identity element of $\G$.
\end{definition}

\begin{definition}[Strong external difference family]
Let $\G$ be a group of order $n$.
An {\bf $(n,m,k,\lambda)$-strong external difference family} 
(or {\bf $(n,m,k,\lambda)$-SEDF})
is a set of $m \geq 2$ disjoint $k$-subsets of $\G$, say $A_1, \dots , A_m$, such that, for every $i$, $1 \leq i \leq m$, the multiset
\begin{equation*}
M_i=\{ xy^{-1}: x \in A_i, y \in\cup_{j \neq i} A_j \}
\end{equation*}
comprises $\lambda$ occurrences of each non-identity element of $\G$.  (An $(n,m,k,\lambda)$-SEDF is, by definition, an $(n,m,k,m\lambda)$-EDF.)
\end{definition}

For every group $\G$, there is at least one SEDF.  This has $k=1$ and is often referred to as the \emph{trivial} SEDF. 
\begin{example}\label{singletons}
Let $\G=\{ g_1, \ldots, g_n \}$ be a group of order $n$.  Then $A_i=\{g_i\}$ for $1 \leq i \leq n$ is an $(n,n,1,1)$-SEDF over $\G$.
\end{example}

To our knowledge, the existing literature contains no examples of non-trivial non-abelian SEDFs.  The only known explicit construction for non-abelian EDFs is given in \cite{HuPa2}; these EDFs have the \emph{bimodal} property and are not SEDFs.   In \cite{Bu1}, new existence results on disjoint difference families (DDFs) in non-abelian groups are established; since $(n,k,k-1)$-DDFs correspond to certain \emph{near-complete} EDFs (see \cite{ChDi} for the proof in the abelian setting), these should guarantee further non-abelian EDFs, but analysis of parameters shows that these will not be strong.

An EDF may be viewed as one generalization of a SEDF; a different generalization is a GSEDF,  defined in \cite{PaSt} and further explored in  \cite{LuNiCa}.  We extend this also to the non-abelian setting.

\begin{definition}[Generalized strong external difference family]
Let $\G$ be a group of order $n$.
An {\bf $(n,m;k_1, \ldots, k_m;\lambda_1, \ldots, \lambda_m)$-generalized strong external difference family} 
(or {\bf $(n,m;k_1, \ldots, k_m;\lambda_1, \ldots, \lambda_m)$-GSEDF})
is a set of $m \geq 2$ disjoint $k_i$-subsets of $\G$, say $A_1, \dots , A_m$, such that, for every $i$, $1 \leq i \leq m$, the multiset
\begin{equation*}
M_i=\{ xy^{-1}: x \in A_i, y \in\cup_{j \neq i} A_j \}
\end{equation*}
comprises $\lambda_i$ occurrences of each non-identity element of $\G$.  (An $(n,m;k, \ldots, k;\lambda, \ldots, \lambda)$-GSEDF is, by definition, an $(n,m,k,\lambda)$-SEDF.)
\end{definition}

\section{Equivalence for EDFs and SEDFs}

In order to enumerate and compare EDFs and SEDFs, it is essential to have a notion of what it means for two such objects to be equivalent.  Equivalence of EDFs has not been previously discussed in the literature; we will adopt a definition consistent with the definition of equivalence for difference families and difference sets (see the Handbook of Combinatorial Designs, Part VI, 16.93, 18.8 \cite{CoDi}).

In preparation for this definition, we establish the following:

\begin{proposition}\label{translate}
Let $\G$ be a group and let $g \in G$.
\begin{itemize}
\item[(i)] If $\A=\{A_1, \ldots, A_m\}$ is an $(n,m,k,\lambda)$-EDF over $\G$,  then the left and right translates $g \A$ and $\A g$ of $\A$ also form $(n,m,k,\lambda)$-EDFs over $\G$.
\item[(ii)] If $\A=\{A_1, \ldots, A_m\}$ is an $(n,m,k,\lambda)$-SEDF over $\G$,  then the left and right translates $g\A$ and $\A g$ of $\A$ also form $(n,m,k,\lambda)$-SEDFs over $\G$.
\end{itemize}
\end{proposition}
\begin{proof}
First consider right translates: by the reversal rule for inverses, $(xg)(yg)^{-1}=(xg)(g^{-1}y^{-1})=x(g g^{-1})y^{-1}=xy^{-1}$  for $x,y \in \G$ .  So for any $A_i g (A_j g)^{-1}$ ($i \neq j$), the multiset of differences is precisely the same as the multiset corresponding to $A_i (A_j)^{-1}$.  For left translates, $(gx)(gy)^{-1}=gx y^{-1} g^{-1}$, the conjugate of $xy^{-1}$ by $g^{-1}$.   So the multiset of differences corresponding to $(g A_i)(g A_j)^{-1}$ is the set of conjugates of the differences $A_i (A_j)^{-1}$ by (fixed) $g^{-1} \in \G$.  For an EDF $\A$, the union of multisets of such differences for all $i \neq j$ comprises each non-identity group element $\lambda$ times; since conjugation by $g^{-1}$ is an automorphism of the group $\G$ (and in particular fixes the identity), the corresponding union for $g \A$ also comprises each non-identity group element $\lambda$ times. An analogous argument holds if $\A$ is an SEDF.
\end{proof}

\begin{proposition}\label{auto}
Let $\G$ be a group, let $\A=\{A_1, \ldots, A_m\}$ be an $(n,m,k,\lambda)$-EDF (respectively, SEDF) over $\G$ and let $\alpha$ be an automorphism of $\G$, Then $\alpha(\A)=\{ \alpha(A_1), \ldots, \alpha(A_m)\}$ forms an $(n,m,k,\lambda)$-EDF  (respectively, SEDF) over $\G$.  
\end{proposition}
\begin{proof}
For $x \in \alpha(A_i)$ and $y \in \alpha(A_j)$, $j \neq i$, consider the difference $xy^{-1}$.  Here $x=\alpha(a_i)$ and $y=\alpha(a_j)$ for some $a_i \in A_i, a_j \in A_j$. Then $xy^{-1}=\alpha(a_i) (\alpha(a_j))^{-1}=\alpha(a_i {a_j}^{-1})$ since $\alpha$ is an automorphism.  For an EDF $\A$, the union of multisets of external differences $A_i(A_j)^{-1}$ for all $i \neq j$ comprises each non-identity group element $\lambda$ times.  Since $\alpha$ is an automorphism of the group $\G$, the union of multisets of external differences $\alpha(A_i) (\alpha(A_j))^{-1}=\alpha(A_i A_j^{-1}) \,(i \neq j)$ also comprises each non-identity group element $\lambda$ times.  An analogous argument holds if $\A$ is an SEDF.
\end{proof}

We now make the following definition:

\begin{definition}\label{equivalent}
\begin{itemize}
\item[(i)]
Let $\G$ be a group of order $n$, and let $\A=\{A_1, \ldots, A_m\}$ and $\A^{\prime}=\{A^{\prime}_1, \ldots, A^{\prime}_m\}$ be two $(n,m,k,\lambda)$-EDFs (respectively, SEDFs) over $\G$.  We shall say that $\A$ is \emph{equivalent} to $\A^{\prime}$ if there is an automorphism $\alpha$ of $G$ and elements $g,h \in \G$ such that, for all $1 \leq i \leq m$,  $A^{\prime}_i=h(\alpha(A_i))g$.
\item[(ii)]
Two $(n,m,k,\lambda)$-EDFs (respectively, SEDFs)  $\A=\{A_1, \ldots, A_m\}$ (in group $\G_1$) and $\A^{\prime}=\{A^{\prime}_1, \ldots, A^{\prime}_m\}$ (in group $\G_2$) are \emph{equivalent} if there is an isomorphism $\alpha$  between $\G_1$ and $\G_2$ and elements $g,h \in \G_2$ such that, for all $1 \leq i \leq m$, $A^{\prime}_i= h(\alpha(A_i))g$.
\end{itemize}
\end{definition}

Propositions \ref{translate} and \ref{auto} guarantee that, if $A$ is an  $(n,m,k,\lambda)$-EDF (respectively, SEDF) and $A^{\prime}$ is equivalent to $A$, then $A^{\prime}$ is also an  $(n,m,k,\lambda)$-EDF (respectively, SEDF).  This definition can be extended to GSEDFs (and other EDF-like structures) in the natural way.

\begin{example}\label{equiv:ex}
\begin{itemize}
\item[(i)] Let $\G=(\mathbb{Z}_{5}, +)$.  The $(5,2,2,1)$-SEDF $(\{0,1\},\{2,4\})$ is equivalent to the SEDF $(\{1,4\}, \{2,3\})$ by translating by $2$, and to the SEDF $(\{0,2\}, \{3,4\})$ by applying the automorphism $x \mapsto 2x$.
\item[(ii)]  Let $\G=(\mathbb{Z}_{17}, +)$.  The $(17,2,4,1)$-SEDF $(\{0,1,4,5\},\{6,8,14,16\})$ is equivalent to the SEDF  $(\{1,4,13,16\},\{2,8,9,15\})$  by applying the automorphism $x \mapsto 3x$ then translating by $1$.
\end{itemize}
\end{example}

\section{Non-existence and existence in general groups}

The following theorem expresses necessary relationships between the parameters of an EDF or SEDF.  Although originally established (\cite{OgKuStSi},\cite{PaSt}) in the context of abelian groups, the conditions remain valid in any finite group.

\begin{proposition}\label{par_rel}
\begin{itemize}
\item[(1)]
Necessary conditions for the existence of an $(n,m,k,\lambda)$-EDF  in a group $\G$ of order $n$ are
\begin{itemize}
\item[(i)]  $m \geq 2$;
\item[(ii)] $n \geq mk$;
\item[(iii)] $ \lambda(n-1)=k^2 m(m-1)$.
\end{itemize}
\item[(2)]Necessary conditions for the existence of an $(n,m,k,\lambda)$-SEDF in a group $\G$ of order $n$ are
\begin{itemize}
\item[(i)] $m \geq 2$;
\item[(ii)] $n \geq mk$;
\item[(iii)] $ \lambda(n-1)=k^2 (m-1)$.
\end{itemize}
\end{itemize}
Parameter sets that satisfy these conditions will be called \emph{admissible}.
\end{proposition}
\begin{proof}
In each case, (i) is in the definition, while (ii) is due to the fact that the $m$ $k$-sets are disjoint.  Part (iii) follows from double-counting the ordered pairs which correspond to the external differences; this does not assume commutativity of the group operation.
\end{proof}

Analysis of the conditions of Proposition \ref{par_rel} can be used to rule-out certain sets of parameters for SEDFs.  Since the analysis is purely number theoretical on the parameter equation, these results apply equally to abelian and non-abelian groups.  The following results from the literature hold for any group of order $n$.
\begin{lemma}\label{lambdalem}
\begin{itemize}
\item For an $(n,m,k,\lambda)$-SEDF, $k=1$ and $\lambda=1$; or $k>1$ and $\lambda<k$ (\cite{HuPa});
\item If $\mathrm{gcd}(k,n-1)=1$, then a non-trivial $(n,m,k,\lambda)$-SEDF does not exist (\cite{JeLi}).
\end{itemize}
\end{lemma}

The following new result significantly restricts the orders of SEDF-containing groups.  
An integer is said to be \emph{square-free} if it is divisible by no square other than $1$.
\begin{proposition}\label{square-free}
Let $n-1$ be square-free.  Then there exists no non-trivial $(n,m,k,\lambda)$-SEDF.
\end{proposition}
\begin{proof}
Let $n-1$ be square-free.  Suppose there exists an $(n,m,k,\lambda)$-SEDF.
By Proposition \ref{par_rel}(iii), we have $\lambda(n-1)=k^2(m-1)$, which rearranges to $n-1=\frac{ k^2(m-1)}{\lambda}$.  Clearly the RHS must be an integer as $n-1$ is, ie $k^2(m-1)$ is divisible by $\lambda$. Moreover, $n-1$ is square-free.  If $p^{\alpha}$ is a prime power divisor of $k$, then it must be a divisor of $\lambda$ (otherwise $\frac{k^2(m-1)}{\lambda}$ would be divisible by $p^2$).  So $\lambda=ak$ for some positive integer $a$, i.e. $k$ divides $\lambda$.   But by Lemma \ref{lambdalem}, for any SEDF we must have either $k=1$ and $\lambda=1$ or $k>1$ and $\lambda<k$.  Since $k \mid \lambda$, we have $k \leq \lambda$ and so it is not possible to have $k>1$. So $k=1$ and $\lambda=1$, ie $n=m$, and the SEDF is trivial.
\end{proof}

We prove a partial converse of Proposition \ref{square-free}.

\begin{proposition}\label{squareful}
Let $n (>2)$ be such that $n-1$ is not square-free.  Then there exists a set of admissible parameters $(n,m,k,\lambda)$ for a non-trivial SEDF.
\end{proposition}
\begin{proof}
Write $n-1=a^2 \cdot b$, where $a^2$ is the largest square dividing $n-1$ and $b \geq 1$. By definition, $a>1$. Take the tuple of parameters $(n,m,k,\lambda)=(n, b+1, a, 1)$: we will  show that this satisfies all three conditions of Proposition  \ref{par_rel} (2), and the non-trivial condition $k>1$.  Taking $\lambda=1$ reduces the parameter equation from Proposition \ref{par_rel} (2)(iii) to $n-1=k^2(m-1)$, ie $a^2b=k^2(m-1)$.  We may set $a=k$ and $b=m-1$ to obtain results which satisfy this equation.  From above $k=a>1$.  Since $b \geq 1$, $m \geq 2$. We must check that $n \geq km$; in fact we prove that $n-1=a^2 b \geq km= a(b+1)$.  We have
$$ a^2b \geq a(b+1) \Leftrightarrow ab \geq b+1  \Leftrightarrow b(a-1)  \geq 1$$
which holds since $a \geq 2$ and $b \geq 1$.
\end{proof}

For a valid group order $n$, there may of course be other sets of admissible parameters for an $(n,m,k, \lambda)$-SEDF, beyond those in the above proof.   Table \ref{table1} contains all admissible SEDF parameter sets for $n$ up to $64$. These were found computationally. 

	\begin{table}[h]
		\centering
		\begin{tabular}{cccc|cccc|cccc|cccc|cccc}
		$n$    & $m$   & $k$    & $\lambda$ & 
		$n$    & $m$    & $k$    & $\lambda$ & 
		$n$    & $m$    & $k$    & $\lambda$ & 
		$n$    & $m$    & $k$    & $\lambda$ & 
		$n$     & $m$    & $k$    & $\lambda$  \\ 
		\hline
		5  & 2 & 2  & 1 & 26 & 3 & 5  & 2 & 37 & 5  & 3  & 1  & 49 & 10 & 4  & 3  & 55                   & 3                    & 9                    & 3                     \\ 
		\cline{1-4}
		9  & 2 & 4  & 2 & 26 & 4 & 5  & 3 & 37 & 5  & 6  & 4  & 49 & 13 & 2  & 1  & 55                   & 3                    & 18                   & 12                    \\ 
		\cline{13-16}
		9  & 3 & 2  & 1 & 26 & 5 & 5  & 4 & 37 & 6  & 6  & 5  & 50 & 2  & 7  & 1  & 55                   & 4                    & 6                    & 2                     \\ 
		\cline{1-8}
		10 & 2 & 3  & 1 & 28 & 2 & 9  & 3 & 37 & 9  & 3  & 2  & 50 & 2  & 14 & 4  & 55                   & 4                    & 12                   & 8                     \\
		10 & 3 & 3  & 2 & 28 & 3 & 9  & 6 & 37 & 10 & 2  & 1  & 50 & 2  & 21 & 9  & 55                   & 5                    & 9                    & 6                     \\ 
		\cline{1-4}\cline{9-12}
		13 & 2 & 6  & 3 & 28 & 4 & 3  & 1 & 41 & 2  & 20 & 10 & 50 & 3  & 7  & 2  & 55                   & 7                    & 3                    & 1                     \\
		13 & 4 & 2  & 1 & 28 & 4 & 6  & 4 & 41 & 3  & 10 & 5  & 50 & 3  & 14 & 8  & 55                   & 7                    & 6                    & 4                     \\ 
		\cline{1-4}
		17 & 2 & 4  & 1 & 28 & 7 & 3  & 2 & 41 & 6  & 4  & 2  & 50 & 4  & 7  & 3  & 55                   & 13                   & 3                    & 2                     \\ 
		\cline{5-8}\cline{17-20}
		17 & 2 & 8  & 4 & 29 & 2 & 14 & 7 & 41 & 11 & 2  & 1  & 50 & 5  & 7  & 4  & 57                   & 2                    & 28                   & 14                    \\ 
		\cline{9-12}
		17 & 3 & 4  & 2 & 29 & 8 & 2  & 1 & 45 & 2  & 22 & 11 & 50 & 6  & 7  & 5  & 57                   & 3                    & 14                   & 7                     \\ 
		\cline{5-8}
		17 & 4 & 4  & 3 & 33 & 2 & 8  & 2 & 45 & 12 & 2  & 1  & 50 & 7  & 7  & 6  & 57                   & 8                    & 4                    & 2                     \\ 
		\cline{9-12}\cline{13-16}
		17 & 5 & 2  & 1 & 33 & 2 & 16 & 8 & 46 & 2  & 15 & 5  & 51 & 2  & 10 & 2  & 57                   & 15                   & 2                    & 1                     \\ 
		\cline{1-4}\cline{17-20}
		19 & 2 & 6  & 2 & 33 & 3 & 4  & 1 & 46 & 3  & 15 & 10 & 51 & 2  & 20 & 8  & 61                   & 2                    & 30                   & 15                    \\
		19 & 3 & 3  & 1 & 33 & 3 & 8  & 4 & 46 & 6  & 3  & 1  & 51 & 3  & 5  & 1  & 61                   & 4                    & 10                   & 5                     \\
		19 & 3 & 6  & 4 & 33 & 4 & 8  & 6 & 46 & 6  & 6  & 4  & 51 & 3  & 10 & 4  & 61                   & 6                    & 6                    & 3                     \\
		19 & 5 & 3  & 2 & 33 & 5 & 4  & 2 & 46 & 11 & 3  & 2  & 51 & 3  & 15 & 9  & 61                   & 16                   & 2                    & 1                     \\ 
		\cline{1-4}\cline{9-12}\cline{17-20}
		21 & 2 & 10 & 5 & 33 & 7 & 4  & 3 & 49 & 2  & 12 & 3  & 51 & 4  & 10 & 6  & 64                   & 2                    & 21                   & 7                     \\
		21 & 6 & 2  & 1 & 33 & 9 & 2  & 1 & 49 & 2  & 24 & 12 & 51 & 5  & 5  & 2  & 64                   & 3                    & 21                   & 14                    \\ 
		\cline{1-8}
		25 & 2 & 12 & 6 & 37 & 2 & 6  & 1 & 49 & 3  & 12 & 6  & 51 & 5  & 10 & 8  & 64                   & 8                    & 3                    & 1                     \\
		25 & 3 & 6  & 3 & 37 & 2 & 12 & 4 & 49 & 4  & 4  & 1  & 51 & 7  & 5  & 3  & 64                   & 8                    & 6                    & 4                     \\
		25 & 4 & 4  & 2 & 37 & 2 & 18 & 9 & 49 & 4  & 8  & 4  & 51 & 9  & 5  & 4  & 64                   & 15                   & 3                    & 2                     \\ 
		\cline{13-16}\cline{17-20}
		25 & 7 & 2  & 1 & 37 & 3 & 6  & 2 & 49 & 4  & 12 & 9  & 53 & 2  & 26 & 13 & \multicolumn{1}{l}{} & \multicolumn{1}{l}{} & \multicolumn{1}{l}{} & \multicolumn{1}{l}{}  \\ 
		\cline{1-4}
		26 & 2 & 5  & 1 & 37 & 3 & 12 & 8 & 49 & 5  & 6  & 3  & 53 & 14 & 2  & 1  & \multicolumn{1}{l}{} & \multicolumn{1}{l}{} & \multicolumn{1}{l}{} & \multicolumn{1}{l}{}  \\ 
		\cline{13-16}
		26 & 2 & 10 & 4 & 37 & 4 & 6  & 3 & 49 & 7  & 4  & 2  & 55 & 2  & 18 & 6  & \multicolumn{1}{l}{} & \multicolumn{1}{l}{} & \multicolumn{1}{l}{} & \multicolumn{1}{l}{}
		\end{tabular}
		\caption{All admissible SEDF parameter sets for orders 1 to 64}
\label{table1}
	\end{table}

It is not the case that an SEDF will necessarily exist for all admissible parameters. Some parameter sets can be ruled-out in certain classes of group using combinatorial or algebraic arguments.  The following result summarises non-existence results for SEDFs in abelian groups:
\begin{proposition}\label{nonexistence}
A non-trivial abelian $(n,m,k,\lambda)$-SEDF does not exist in each of the following cases:
\begin{itemize}
\item[(1)] $m \in \{3,4\}$ (\cite{MaSt});
\item[(2)] $m>2$ and $n$ prime (\cite{MaSt});
\item[(3)] $m>2$ and $n=pq$ for distinct primes $p,q$ (\cite{BaJiWeZh});
\item[(4)] $m>2$ and $\lambda=2$ (\cite{HuPa});
\item[(5)] $m>2$ and $\lambda>1$ and $\frac{\lambda(k-1)(m-2)}{(\lambda-1)k(m-1)}>1$ (\cite{HuPa});
\item[(6)] $m>2$ and there is a prime $p$ dividing $n$ such that $\mathrm{gcd}(km,p)=1$ and $m$ is not congruent to $2$ mod $p$  (\cite{BaJiWeZh});
\item [(7)] $m>2$ and $n$ a prime power, when $\G$ is cyclic (\cite{BaJiWeZh} and \cite{LeLiPr});
\item[(8)] $m>2$ and $n$ is the product of at most three (not necessarily distinct) primes, except possibly when $\G=C_p^3$ and $p$ is a prime greater than $3 \times 10^{12}$ (\cite{LeLiPr}).
\end{itemize}
\end{proposition}
Of these, (4) and (5) employ direct combinatorial arguments which rely on the commutativity of the operation, while the others use character theory in abelian groups.   It is therefore possible that SEDFs with these forbidden parameters may exist in non-abelian groups.  

The following is a summary of the main constructive existence results known for abelian groups.  All known constructions for families of SEDFs have $m=2$.

\begin{proposition} \label{SEDFexistence}
An $(n,m,k, \lambda)$-SEDF exists in the group $\G$ in the following cases:
\begin{itemize}
\item[(1)] $(n,m,k,\lambda)=(k^2+1, 2, k, 1)$ and $\G=\mathbb{Z}_{k^2+1}$ (\cite{PaSt}); sets given by $\{0,1,\ldots,k-1 \}, \{k,2k,\ldots,k^2\}$.  
\item[(2)] $(n,m,k,\lambda)=(n, 2, \frac{n-1}{2}, \frac{n-1}{4})$ and $n$ is congruent to $1$ mod $4$, provided there exists an $(n,\frac{n-1}{2},\frac{n-5}{4},\frac{n-1}{4})$ partial difference set in $\G$ (\cite{HuPa}). 
 In particular, when $n$ is a prime power, appropriate sets are given by the non-zero squares and non-squares in the multiplicative group of the finite field of order $n$.
\item[(3)] $(n,m,k,\lambda)=(q,2,\frac{q-1}{4},\frac{q-1}{16})$, where $q=16 t^2+1$ is a prime power and $t$ an integer (\cite{BaJiWeZh}); the sets are cyclotomic classes in the finite field of order $q$.
\item[(4)] $(n,m,k,\lambda)=(p,2,\frac{p-1}{6},\frac{p-1}{36})$, where $p=108 t^2+1$ is a prime and $t$ an integer  (\cite{BaJiWeZh}); the sets are cyclotomic classes in the finite field of order $p$.
\end{itemize}
\end{proposition}
Note that different constructions may produce equivalent SEDFs for certain specific parameter sets.  For example, when $n=5$,  Example \ref{equiv:ex} shows that Proposition \ref{SEDFexistence} (1) and (2) yield equivalent SEDFs, although this is not the case for larger values of $n$.

For $\lambda=1$, a parameter characterization for the abelian case is given in \cite{PaSt}:
\begin{proposition}\label{PaStlambda=1}
 There exists an abelian $(n,m,k,1)$-SEDF if and only if $m=2$ and $n=k^2+1$ or $k=1$ and $m=n$.
\end{proposition}  
The proof of the the forward direction relies on the commutativity of the operation, while for the reverse direction, constructions are provided by the trivial SEDF and Proposition \ref{SEDFexistence} (1).

Finally, we mention the following characterization for the \emph{near-complete} abelian case (when $n=mk+1$), given in \cite{JeLi}.  An alternative construction for a $(243,11,22,20)$-SEDF, using cyclotomic methods, is given in \cite{WeYaFuFe}.
\begin{proposition}\label{nearcomplete}
Let $\A=\{ A_1, \ldots, A_m \}$ be a collection of $k$-sets ($k>1$) which partition the non-identity elements of an abelian group $\G$.  Then $\A$ is a non-trivial $(n,m,k, \lambda)$-SEDF if and only if
\begin{itemize}
\item[(i)] $(n,m,k,\lambda)=(n,2,\frac{n-1}{2}, \frac{n-1}{4})$, $n$ is congruent to $1$ mod $4$, and $A_1$ is a non-trivial regular $(n,\frac{n-1}{2}, \frac{n-5}{4}, \frac{n-1}{4})$ partial difference set in $\G$; or
\item[(ii)] $(n,m,k,\lambda)=(243,11,22,20)$ and each $A_j$ is a non-trivial regular $(243,22,1,2)$ partial difference set in $\G$ for $1 \leq j \leq 11$.
\end{itemize}
\end{proposition}

We note that $(243,11,22,20)$ are the only parameters with $m>2$ for which a SEDF is known to exist.



\section{Abelian SEDFs with $\lambda=1$: recursive constructions and equivalence}

In this section, we consider the situation when $\lambda=1$.  By Proposition \ref{PaStlambda=1}, all non-trivial abelian $(n,m,k,1)$-SEDFs are $(k^2+1,2,k,1)$-SEDFs.  We show that the known construction of  Proposition \ref{SEDFexistence}(1) is actually one amongst numerous different constructions for SEDFs with $\lambda=1$ in cyclic groups, and we provide recursive techniques to obtain these (which also provide constructions for GSEDFs).  Moreover, we demonstrate the existence of at least two non-equivalent $(k^2+1,2, k, 1)$-SEDFs for any composite $k$ (in the subsequent section we will extend this to any $k>2$).

Throughout, we will write  $\mathbb{Z}_n$ additively; its elements will be $\{0,1,\ldots,n-1\}$ and we fix the natural ordering $0< \cdots < n-1$.

We begin by presenting a new explicit construction for a $(k^2+1,2,k,1)$-SEDF in $(\mathbb{Z}_{k^2+1}, +)$ for even $k$, and show it is not equivalent to the known construction for $k>2$.  Part of its subtraction table is shown in Table \ref{tableNewZ}; the entry in the row labelled by $x$ and the column labelled by $y$ is $x-y$.

\begin{theorem}\label{newZconstruction}
 Let $\G=(\mathbb{Z}_{k^2+1},+)$, where $k=2a$ for some positive integer $a \geq 1$.  
 \begin{itemize}
 \item[(i)] Let $B_1=\{0,1, \ldots, a-1\} \cup \{2a, 2a+1, \ldots 3a-1\}$ and let $B_2= \bigcup_{i=1}^a \{ (4i-1)a, 4ia\}$; then $\{B_1, B_2\}$ is a $(k^2+1,2,k,1)$-SEDF in $\G$.
\item[(ii)] For $k>2$,  $\{ B_1, B_2\}$ is not equivalent to the SEDF in $\mathbb{Z}_{k^2+1}$ obtained from the construction of Proposition \ref{SEDFexistence} (1). 
 \end{itemize}
\end{theorem}
\begin{proof}
(i) The multiset $B_2-B_1$ is the union $\cup_{i=1}^a (D_i \cup E_i)$ where
\begin{itemize}
\item $D_i=\{ (4i-1)a, 4ia\}-\{2a,2a+1,\ldots,3a-1\}=\{(4i-4)a+1, (4i-4)a+2, \ldots, (4i-2)a\}$
\item $E_i=\{ (4i-1)a, 4ia \} - \{ 0,1, \ldots, a-1 \}=\{ (4i-2)a+1, (4i-2)a+2, \ldots, 4ia \}$.
\end{itemize}
Clearly this comprises each non-zero element of $\mathbb{Z}_{4a^2+1}$ precisely once.  The same is necessarily true for the multiset $B_1-B_2$, which consists of the negatives of these elements.\\

(ii)   Let $S_1$ be the $(k^2+1,2,k,1)$-SEDF in $\G$ with sets $A_1=\{0, 1,\ldots, 2a-2, 2a-1\}$ and $A_2=\{2a, 4a, \ldots, (4a-2)a, 4a^2 \}$ from Proposition \ref{SEDFexistence} (1).  Denote by $S_2$ the SEDF $\{B_1, B_2 \}$ from part (i).  We shall show that, for $k>2$, $S_1$ and $S_2$ are not equivalent.
 
 Recall that the automorphisms of $\mathbb{Z}_n$ are of the form $x \mapsto cx$ where $c$ is a unit of $\mathbb{Z}_n$.  By Definition \ref{equivalent}, $S_1$ is equivalent to $S_2$ if there exists a function $f(x)=\alpha x + \beta$, with $\alpha \in \mathbb{Z}_{k^2+1}^*$ (the group of units of $\mathbb{Z}_{k^2+1}$) and $\beta \in \mathbb{Z}_{k^2+1}$, such that $f(S_1)=S_2$. We prove that no such $f$ can exist for $a>1$.  Observe that for $a>1$, each block of each SEDF contains at least four elements.

Aiming for a contradiction, assume that there exists such an $f$.  First suppose that $f$ maps $A_1$ onto $B_1$.  Fix an ordering of the elements of $A_1$ as they are listed in the theorem statement.  We show that the images under $f$ of any sequence $\{x,x+1,x+2\}$ of three consecutive elements of $A_1$ must all lie in $C_1=\{0,1,\ldots, a-1\}$ or all lie in $C_2=\{2a, \ldots, 3a-1\}$.  Observe that $f(x)+f(x+2)=(\alpha x + \beta)+(\alpha(x+2)+\beta)=\alpha(2x+2)+ 2 \beta= 2 f(x+1)$.  Suppose that one of $f(x)$ and $f(x+2)$ lies in $C_1$ and the other in $C_2$.  Then $2a \leq f(x)+f(x+2) \leq 4a-2 (< 4a^2+1)$, ie $a \leq f(x+1) \leq 2a-1$; this is impossible since $B_1=\{0,1, \ldots, 3a-1\} \setminus \{a, \ldots, 2a-1\}$.  So both of $x$ and $x+2$ are mapped to the same $C_i$ ($i \in \{1,2\}$).  If this is $C_1$, then $1 \leq f(x)+f(x+2) \leq 2a+3 (<4a^2+1)$, i.e. $1 \leq f(x+1) < a-2$, so $f(x+1) \in C_1$.  If this is $C_2$, then $4a+1 \leq f(x)+f(x+2) \leq 6a-3 (<4a^2+1)$, i.e. $1 \leq f(x+1) < 3a-2$, so $f(x+1) \in C_2$. Thus all three elements are mapped to the same $C_i$.

However, repeated applications of this result to $\{0,1,2\}, \{1,2,3\}, \ldots, \{2a-3, 2a-2, 2a-1\}$ in $A_1$ shows that all $2a$ elements of $A_1$ must be mapped to the $a$ elements of one of the $C_i$, which is impossible.  Hence $f$ cannot map $A_1$ onto $B_1$.

Now, suppose $f$ maps $A_2$ onto $B_1$.  Observe that $A_2=2a(A_1+1)$.  Since $f(A_2)=B_1$, we have $B_1=f(A_2)=\alpha A_2 + \beta = 2a \alpha A_1 + (2 a \alpha + \beta)$.  Now, $2a$ is a unit of $\G$ (with inverse $-2a)$ and $\alpha$ is a unit by definition, so $2a \alpha$ is a unit.  Thus there is a function $g(x)= 2a\alpha x+ (2 a \alpha + \beta)$ such that $2a\alpha$ is a unit, $2a \alpha+\beta$ is a group element and $g(A_1)=B_1$.  But we have proved above that this cannot occur; again a contradiction.  So no such $f$ exists, and $S_1$ and $S_2$ are non-equivalent.
\end{proof}

\begin{table}[h]
    \renewcommand{\arraystretch}{1.2}
		\centering
		\setlength\tabcolsep{3pt}
		\begin{tabular}{|c|cccc|cccc|} 
			\hline
			& $0$ & $1$ & $\ldots$ & $a-1$ & $2a$ & $2a+1$ & $\ldots$ & $3a-1$	\\ \hline
			$3a$ & $3a$ & $3a-1$ & $\ldots$ & $2a+1$ & $a$ & $a-1$ & $\ldots$ & $1$	\\
			$4a$ & $4a$ & $4a-1$ & $\ldots$ & $3a+1$ & $2a$ & $2a-1$ & $\ldots$ & $a+1$	\\ \hline
			$7a$ & $7a$ & $7a-1$ & $\ldots$ & $6a+1$ & $5a$ & $5a-1$ & $\ldots$ & $4a+1$	\\
			$8a$ & $8a$ & $8a-1$ & $\ldots$ & $7a+1$ & $6a$ & $6a-1$ & $\ldots$ & $5a+1$\\ \hline
			$\vdots$ & $\vdots$ & $\ldots$ & $\ldots$ & $\vdots$ & $\vdots$ & $\ldots$ & $\ldots$ & $\vdots$\\ \hline
			$(4i-1)a$ & $(4i-1)a$ & $(4i-1)a-1$ & $\ldots$ & $(4i-2)a+1$ & $(4i-3)a$ & $(4i-3)a-1$ & $\ldots$ & $(4i-4)a+1$	\\
			$4ia$ & $4ia$ & $4ia-1$ & $\ldots$ & $(4i-1)a+1$ & $(4i-2)a$ & $(4i-2)a-1$ & $\ldots$ & $(4i-3)a+1$\\ \hline
			$\vdots$ & $\vdots$ & $\ldots$ & $\ldots$ & $\vdots$ & $\vdots$ & $\ldots$ & $\ldots$ & $\vdots$\\ \hline
			$(4a-1)a$ & $(4a-1)a$ & $(4a-1)a-1$ & $\ldots$ & $(4a-2)a+1$ & $(4a-3)a$ & $(4a-3)a-1$ & $\ldots$ & $(4a-4)a+1$	\\
			$4a^2$ & $4a^2$ & $4a^2-1$ & $\ldots$ & $(4a-1)a+1$ & $(4a-2)a$ & $(4a-2)a-1$ & $\ldots$ & $(4a-3)a+1$\\ \hline
		\end{tabular}
	
		\caption{Table showing $B_2-B_1$ for the construction in Theorem \ref{newZconstruction}}
		\label{tableNewZ}
	
	\end{table}

\begin{example}\label{newZex}
Using the construction of Theorem \ref{newZconstruction}, 
\begin{itemize}
\item a $(17,2,4,1)$-SEDF in $\mathbb{Z}_{17}$ is given by $B_1=\{0,1,4,5\}$, $B_2=\{6,8,14,16\}$; 
\item a $(37,2,6,1)$-SEDF in $\mathbb{Z}_{37}$ is given by $B_1=\{0,1,2,6,7,8\}$, $B_2=\{9,12,21,24,33,36\}$.  
\end{itemize}
\end{example}
Observe that $k=2$ is excluded from the above theorem, since in this case we obtain a $(5,2,2,1)$-SEDF in $\mathbb{Z}_5$ with $B_1=\{0,2\}$, $B_2=\{3,4\}$.  By Example \ref{equiv:ex}, this is equivalent to the $(5,2,2,1)$-SEDFs constructed using Proposition \ref{SEDFexistence} (1) and (2).

We have the following recursive result for SEDFs with $\lambda=1$ in cyclic groups, which encompasses both previous constructions and provides a means of generating new examples.

\begin{theorem}\label{recursive}
Let $S=\{A_1, A_2\}$ be a $(k^2+1,2,k,1)$-SEDF in $\mathbb{Z}_{k^2+1}$ with $A_1=\{x_1, \ldots, x_k\}$ and $A_2=\{y_1, \ldots, y_k\}$, such that $x_i<y_j$ for all $1 \leq i,j \leq k$.  Let $a \in \mathbb{N}$.  Then we can obtain from $S$ an $((ak)^2+1, 2, ak,1)$-SEDF $S^{\prime}=\{B_1,B_2\}$ in $\mathbb{Z}_{(ak)^2+1}$.  The blocks are given by
\begin{itemize}
\item $B_1= \bigcup_{i=1}^k \{ a x_i  + \alpha: 0 \leq \alpha \leq a-1 \}$
\item $B_2=\bigcup_{i=1}^k \{ a(y_i+ k^2 \beta) : 0 \leq \beta \leq a-1 \}$.
\end{itemize}
(Here $x_i,y_i$ denote the elements of $\mathbb{Z}_{(ak)^2+1}$ with these labels.)
\end{theorem}
\begin{proof}
For a fixed pair of indices $(i,j)$ ($1 \leq i,j \leq k$), consider the multiset of differences
$$ D_{ij}=\{ ay_i - (ax_j + \alpha): 0 \leq \alpha \leq a-1 \}$$
i.e. $\{a(y_i-x_j)+ \alpha: 0 \leq \alpha \leq a-1 \} = \{a(y_i-x_j), a(y_i-x_j)-1, \ldots, a(y_i-x_j-1)+1 \}$.  As $(i,j)$ runs through all possible pairs, $(y_i-x_j)$ takes each value $1, \ldots, k^2$ precisely once, since $S$ is a $(k^2+1,2,k,1)$-SEDF in $\mathbb{Z}_{k^2+1}$ and all differences $y_i-x_j$ lie in the set $\{1, \ldots, k^2\}$ since $x_i<y_j$ for all $1 \leq i,j \leq k$.  Thus $\cup_{i,j} D_{ij}$ contains each element from $1$ to $ak^2$ in $\mathbb{Z}_{(ak)^2+1}$.  

Now, for given $(i,j)$,  the multiset $\{ a(y_i+ k^2 \beta)-(a x_j+\alpha): 0 \leq \alpha \leq a-1\}= ak^2 \beta+D_{ij}$, and so for a fixed $\beta \in \{0, \ldots, a-1\}$, the union $\Delta_{\beta}$ of all these multisets as $(i,j)$ runs through all possible pairs yields all elements from $1+ak^2 \beta$ to $ak^2+ak^2 \beta=ak^2 (\beta+1)$ of $\mathbb{Z}_{(ak)^2+1}$, precisely once each.

Finally, observe that the union $\cup _{\beta=0}^{a-1} \Delta_{\beta}$ comprises all elements of $\mathbb{Z}_{(ak)^2+1}$ from $1$ to $(ak)^2$, once each.  Since this union equals $B_2-B_1$ (and $B_1-B_2$ consists of the negatives of these elements), we have established that $S^{\prime}$ has the required property.
\end{proof}

Table \ref{tableRecursive} illustrates part of the subtraction table for Theorem \ref{recursive}.  The elements of $B_2$ have been re-ordered to follow the proof approach.

\begin{table}[h]
		\centering
		\resizebox{\columnwidth}{!}{
		\begin{tabular}{|c|cccc|c|cccc|} 
			\hline
			& $a x_1$ & $a x_1+1$ & $\ldots$ & $ax_1+(a-1)$ & \ldots & $a x_k$ & $a x_k+1$ & $\ldots$ & $ax_k+(a-1)$ \\ \hline
			$ay_1$ & $a(y_1- x_1)$ & $a (y_1-x_1)-1$ & $\ldots$ & $a(y_1-x_1-1)+1$ & $\ldots$ & $a (y_1-x_k)$ & $a (y_1-x_k)-1$ & $\ldots$ & $a(y_1-x_k-1)+1$ \\
			$\vdots$ & $\vdots$ & $\vdots$ & $\vdots$ & $\vdots$ & $\ldots$ & $\vdots$ & $\vdots$ & $\vdots$ & $\vdots$ \\
			$ay_k$ & $a(y_k- x_1)$ & $a (y_k-x_1)-1$ & $\ldots$ & $a(y_k-x_1-1)+1$ & $\ldots$ & $a (y_k- x_k)$ & $a (y_k-x_k)-1$ & $\ldots$ & $a(y_k-x_k-1)+1$ \\ \hline

			$a(k^2+y_1)$ & $a(k^2+y_1- x_1)$ & $a (k^2+y_1-x_1)-1$ & $\ldots$ & $a(k^2+y_1-x_1-1)+1$ & $\ldots$ & $a (k^2+y_1-x_k)$ & $a (k^2+y_1-x_k)-1$ & $\ldots$ & $a(k^2+y_1-x_k-1)+1$ \\
			$\vdots$ & $\vdots$ & $\vdots$ & $\vdots$ & $\vdots$ & $\ldots$ & $\vdots$ & $\vdots$ & $\vdots$ & $\vdots$ \\
			$a(k^2+y_k)$ & $a(k^2+y_k- x_1)$ & $a (k^2+y_k-x_1)-1$ & $\ldots$ & $a(k^2+y_k-x_1-1)+1$ & $\ldots$ & $a (k^2+y_k- x_k)$ & $a (k^2+y_k-x_k)-1$ & $\ldots$ & $a(k^2+y_k-x_k-1)+1$ \\ \hline
			$\vdots$ & $\vdots$ & $\vdots$ & $\vdots$ & $\vdots$ & $\ldots$ & $\vdots$ & $\vdots$ & $\vdots$ & $\vdots$ \\ \hline
			
			$a((a-1)k^2+y_1)$ & $a((a-1)k^2+y_1- x_1)$ & $\ldots$ & $\ldots$ & $\ldots$ & $\ldots$ & $a ((a-1)k^2+y_1-x_k)$ & $\ldots$ & $\ldots$ & $\ldots$ \\
			$\vdots$ & $\vdots$ & $\vdots$ & $\vdots$ & $\vdots$ & $\ldots$ & $\vdots$ & $\vdots$ & $\vdots$ & $\vdots$ \\
			$a((a-1)k^2+y_k)$ & $a((a-1)k^2+y_k- x_1)$ & $\ldots$ & $\ldots$ & $\ldots$ & $\ldots$ & $a ((a-1)k^2+y_k- x_k)$ & $\ldots$ & $\ldots$ & $\ldots$ \\ \hline
		\end{tabular} }
	
		\caption{Table showing $B_2-B_1$ for the construction in Theorem \ref{recursive}}
		\label{tableRecursive}
	
	\end{table}

\begin{example}
Using the construction of Theorem \ref{recursive}:
\begin{itemize}
\item Let $S$ be the trivial $(2,2,1,1)$-SEDF $A_1=\{0\}$, $A_2=\{1\}$ in $\mathbb{Z}_2$.  Then $S^{\prime}$ is the $(a^2+1,2,a,1)$-SEDF $B_1=\{0,1,\ldots,a-1\}$, $B_2=\{a,2a,\ldots,a^2\}$ in $\mathbb{Z}_{a^2+1}$ of Proposition \ref{SEDFexistence} (1),  from the paper \cite{PaSt}.
\item If $S$ is the $(5,2,2,1)$-SEDF $\{0,1\}, \{2,4\}$ in $\mathbb{Z}_5$, then $S^{\prime}$ is the $(4a^2+1, 2, 2a, 1)$-SEDF in $\mathbb{Z}_{4a^2+1}$ from Proposition \ref{SEDFexistence} (1), whereas if $S$ is $\{0,2\}, \{3,4\}$ in $\mathbb{Z}_5$ then $S^{\prime}$ is the $(4a^2+1, 2, 2a, 1)$-SEDF in $\mathbb{Z}_{4a^2+1}$ from Theorem \ref{newZconstruction}.  This shows that, if the recursion is performed using two equivalent SEDFs as $S$, the resulting $S^{\prime}$'s need not be equivalent.
\end{itemize}
\end{example}

We may now prove the following result, which generalizes Theorem \ref{newZconstruction}.
\begin{theorem}\label{kcomposite}
For every $\mathbb{Z}_{k^2+1}$ with $k$ composite, there are at least two non-equivalent $(k^2+1, 2, k, 1)$-SEDFs. 
\end{theorem}
\begin{proof}
Since $k$ is composite, we can write $k=ra$ for some $a, r \geq 2$.  Consider the following two $(r^2+1, 2, r, 1)$-SEDFs in $\mathbb{Z}_{r^2+1}$:
\begin{itemize}
\item $S_1=\{ \{0,1,\ldots, r-1\}, \{r, 2r, \ldots, r^2\} \}$;
\item $S_2=rS_1=\{ \{0,r,\ldots, r^2-r\}, \{r^2-r+1, \ldots, r^2-1, r^2\} \}$.
\end{itemize}
Each of these satisfies the requirements of Theorem \ref{recursive}, so we may use the recursive construction of Theorem \ref{recursive} with $S_1$ and $S_2$, and $a(\geq 2)$, to obtain two $((ra)^2+1, 2, ra, 1)$-SEDFs (say $T_1$ and $T_2$ respectively).  Although $S_1$ and $S_2$ are themselves equivalent, we prove that $T_1$ and $T_2$ are non-equivalent SEDFs in $\mathbb{Z}_{k^2+1}$.  This generalizes the proof strategy of Theorem \ref{newZconstruction}.

Applying Theorem \ref{recursive} to $S_1$ yields $T_1=\{B_1, B_2\}$, where $B_1=\{0,1,\ldots,ar-1\}$ and $B_2=\{ar, 2ar, \ldots, (ar)^2\}$.  (This is the SEDF that would be obtained by setting $k=ar$ in the construction of Proposition \ref{SEDFexistence} (1).)  Applying Theorem \ref{recursive} to $S_2$ yields $T_2=\{U_1, U_2\}$, where $U_1=\cup_{i=0}^{r-1} F_i$ and $F_i=\{iar, iar+1, \ldots, iar+(a-1)=a(ir+1)-1 \}$.

Recall that the automorphisms of a cyclic group $\mathbb{Z}_n$ are of the form $x \mapsto cx$ where $c$ is a unit of $\mathbb{Z}_n$.  By Definition \ref{equivalent}, $T_1$ is equivalent to $T_2$ if there exists a function $f(x)=\alpha x + \beta$, with $\alpha \in \mathbb{Z}_{k^2+1}^*$ (the group of units of $\mathbb{Z}_{k^2+1}$) and $\beta \in \mathbb{Z}_{k^2+1}$, such that $f(T_1)=T_2$. We prove that no such $f$ can exist.  Observe that, since $a,r \geq 2$, each block of each SEDF contains at least four elements.

Any such $f$ would have to map one of $B_1$ or $B_2$ to $U_1$.  We claim that, in fact, it is sufficient to show that there is no $f$ which maps $B_1$ onto $U_1$.  For, suppose we have established this claim, and suppose there is an $f$ which maps $B_2$ onto $U_1$.  Observe that $B_2=ar(B_1+1)$.  Since $f(B_2)=U_1$, we have $U_1=f(B_2)=\alpha B_2 + \beta = (\alpha ar) B_1 + (\alpha ar + \beta)$.  Now, $ar$ is a unit of $\G$ (with inverse $-ar)$ and $\alpha$ is a unit by definition, so $\alpha ar$ is a unit.  Thus there is a function $g(x)= \alpha ar x+ (\alpha ar+\beta)$ such that $\alpha ar$ is a unit, $\alpha ar+\beta$ is a group element and $g(B_1)=U_1$.  But we have proved above that this cannot occur; a contradiction. 

So we now suppose, aiming for a contradiction, that there is an $f(x)= \alpha x+ \beta$ such that $f(B_1)=U_1=F_0 \cup F_1 \cup \cdots \cup F_{r-1}$.  For any set $\{x, x+1, x+2\}$ of elements of $B_1$, their images under $f$ must satisfy
$$ f(x+1)+f(x+1)=f(x)+f(x+2).$$
Note that all sums of pairs of elements of $U_1$ are strictly less than (under the specified ordering) $2ar^2-1$, i.e. less than $a^2r^2$ (since $a,r \geq 2$), and so do not reduce modulo $a^2 r^2+1$.

Now, let $i,j,k$ be the indices such that $f(x+1) \in F_i$, $f(x) \in F_j$ and $f(x+2) \in F_k$ (here $i,j,k$ are integers with $0 \leq i,j,k, \leq r-1$).  Then
$$ 2iar \leq f(x+1)+f(x+1) \leq 2iar+2a-2< (2i+1)ar $$
since $ar \geq 2a> 2a-2$.  Moreover,
$$ (j+k)ar \leq f(x)+f(x+2) \leq (j+k)ar+2a-2< (j+k+1)ar. $$
Now, if $j+k \geq 2i+1$, then $f(x)+f(x+2) \geq (2i+1)ar$ whereas $f(x+1)+f(x+1) < (2i+1)ar$, and similarly if $j+k \leq 2i-1$ then $f(x)+f(x+2) < 2iar$ whereas $f(x+1)+f(x+1) \geq 2iar$.  So we must have $j+k=2i$.  If any two of $\{i,j,k\}$ are equal, then the third is equal to both.  Hence there are two possibilities: $i=j=k$ or all of $\{i,j,k\}$ are distinct and form a three-term arithmetic progression.  This latter case is possible only when $r>2$.

We now consider the elements $\{0,1,2\} \subset B_1$.  Let $f(0) \in F_j$, $f(1) \in F_i$ and $f(2) \in F_k$.

If $i=j=k$, then all of $\{0,1,2\}$ are mapped to $F_i$ by $f$.  Then $\{1,2,3\} \subset B_1$ has two of its members mapped to $F_i$ and hence also $f(3) \in F_i$.  Repeating this process, we see that all $ar$ elements of $B_1$ are mapped to $F_i$, a contradiction since $|F_i|=a$ and $r \geq 2$.

If $r=2$, the proof is complete.  Otherwise, for $r \geq 3$, it must be the case that all of $\{i,j,k\}$ are distinct.  Since $2i=j+k$, one of $\{j,k\}$ must be less than $i$ and the other greater than $i$.  Suppose first that $k<i<j$.  Let $f(3) \in F_m$ and consider $\{1,2,3\}$; we have $i+m=2k$ and $\{i,k\}$ are distinct.  Thus all three of $\{i,k,m\}$ are distinct and $j-i=i-k=k-m$, ie $m<k<i<j$.  Repeating for all $0,1,\ldots, ar-1 \in B_1$, we deduce that all $ar$ elements must lie in distinct $F_i$'s. But there are only $r$ possible indices, and $a \geq 2$, so this is impossible.  The situation when $j<i<k$ is precisely analogous.
 
Thus all cases lead to a contradiction, and so there does not exist a function $f$ such that $f(B_1)=U_1$.  This completes the proof that $T_1$ and $T_2$ are non-equivalent.
\end{proof}

In fact, the recursive process of Theorem \ref{recursive} can be performed for GSEDFs with $m=2$ and $\lambda_1=\lambda_2=1$, and by appropriate choice of parameters we can build new SEDFs using GSEDFs which are not themselves SEDFs.   A recursive construction for GSEDFs was given in \cite{LuNiCa}; our result differs from this in that the value of $\lambda$ remains unchanged, and the new group is a larger cyclic group rather than a cross-product.

\begin{theorem}\label{GSEDFrecursive}
Let $S=\{A_1, A_2\}$ be a $(st+1,2; s,t; 1,1)$-GSEDF in $\mathbb{Z}_{st+1}$ with $A_1=\{x_1, \ldots, x_s\}$ and $A_2=\{y_1, \ldots, y_t\}$, such that $x_i<y_j$ for all $1 \leq i \leq s$ and $1 \leq j \leq t$.  Let $a,b \in \mathbb{N}$.  Then we can obtain from $S$ an $(abst+1, 2; as, bt; 1,1)$-GSEDF $S^{\prime}=\{B_1,B_2\}$ in $\mathbb{Z}_{abst+1}$.  The blocks are given by
\begin{itemize}
\item $B_1= \cup_{i=1}^s \{ a x_i  + \alpha: 0 \leq \alpha \leq a-1 \}$
\item $B_2=\cup_{i=1}^t \{ a(y_i+ \beta st) : 0 \leq \beta \leq b-1 \}$.
\end{itemize}
(here $x_i,y_i$ denote the elements of $\mathbb{Z}_{abst+1}$ with these labels).
\end{theorem}
\begin{proof}
The proof is precisely analogous to that of Theorem \ref{recursive}.
\end{proof}

\begin{corollary}
Let $k=as=bt$ for some $a,b,s,t \in \mathbb{N}$, such that there exists an $(st+1,2; s,t; 1,1)$-GSEDF $\{A_1, A_2\}$ in $\mathbb{Z}_{st+1}$ with $x<y$ for all $x \in A_1$ and $y \in A_2$.  Then an $(k^2+1, 2; k,1)$-SEDF can be obtained from $S$.
\end{corollary}
An example of a GSEDF construction which can be used in the recursion is given by $S=(\{0,1\ldots,s-1\}, \{s, 2s, \ldots, ts \})$, presented in Theorem 4.4 of \cite{LuNiCa}.  The equivalent GSEDF $T=(\{0,t, \ldots (s-1)t\}, \{(s-1)t+1, (s-1)t+2, \ldots, st\}$ (obtained via the automorphism $x \mapsto tx$) may also be used. 

There is a very wide range of specific SEDF constructions obtainable from Theorem \ref{GSEDFrecursive}.  We illustrate with a specific example.

\begin{example}
\begin{itemize}
\item[(i)] Taking $s=2$ and $t=3$ with $T=(\{0,3\}, \{4,5,6\})$ in $\mathbb{Z}_7$, we obtain the following general construction in $\mathbb{Z}_{k^2+1}$ for any $k$ such that $k=2a=3b$ for some $a,b \in \mathbb{N}$, ie for any $k$ which is a multiple of $6$: 
\begin{itemize}
\item $B_1=\{0,1,\ldots, a-1\} \cup \{3a,3a+1, \ldots, 4a\}$
\item $B_2=\{4a,5a,6a \} \cup \{10a, 11a, 12a\} \cup \cdots \cup \{(6b-2)a, (6b-1)a , 6ab\}$
\end{itemize}
\item[(ii)] Taking $k=12$, i.e. $a=6$ and $b=4$, we obtain the following $(145,2,12,1)$-SEDF in $\mathbb{Z}_{145}$:
\begin{itemize}
\item $B_1=\{0,1,2,3,4,5,18,19,20,21,22,23\}$
\item $B_2=\{24,30,36,60,66,72,96,102,108,132,138,144\}$
\end{itemize}
\end{itemize}
\end{example}

\section{Non-abelian SEDFs with $\lambda=1$}

In the non-abelian setting, it is not  known whether every $(n,m,k,1)$-SEDF must be a $(k^2+1, 2, k, 1)$-SEDF.  The forward direction of Proposition \ref{PaStlambda=1} does not apply here, so it is possible that there may exist non-abelian $(k^2(m-1)+1, m, k, 1)$-SEDFs with $m>2$, although no examples are currently known. However, $(k^2+1, 2, k,1)$ is an admissible parameter set.  Unlike in the abelian case, there does not exist a non-abelian group of order $k^2+1$ for every value of $k$ - for example there exists no non-abelian group of order $k^2+1$ for $k \in \{2,4,6,8,10 \}$.

In this section, we exhibit an infinite family of non-abelian $(k^2+1, 2, k, 1)$-SEDFs for $k$ odd.  

We will consider the dihedral group $D_n$ of order $n$ as being generated by the elements $s$ (reflection) and $r$ (rotation by $\frac{4 \pi}{n}$), so that $D_n=\{ s^i r^j: 0 \leq i \leq 1, 0 \leq j < \frac{n}{2} \}$ and the generators satisfy $s^2=1$, $r^{\frac{n}{2}}=1$ and $sr=r^{-1}s$.

First, we present two examples, which illustrate the general construction to follow. 

\begin{example}\label{dihedralex}

\begin{itemize}
\item[(i)] Let $k=3$ and $\G=D_{10}$; take $A_1=\{e,s,r\}$ and $A_2=\{sr, r^3, sr^4\}$.
\item[(ii)] Let $k=5$ and $\G=D_{26}$; take $A_1=\{e,s,r,sr,r^2\}$ and $A_2=\{sr^2, r^5, sr^7,r^{10}, sr^{12}\}$.
\end{itemize}
\vspace{5mm}
Table \ref{table2} demonstrates the multiset of external differences for the $(26, 2, 5, 1)$-SEDF over $D_{26}$; the entry in the row labelled by $x$ and the column labelled by $y$ is $xy^{-1}$.  The sets have been ordered to emphasise the structure.

\begin{table}[h]
    \renewcommand{\arraystretch}{1.2}
		\centering
		\setlength\tabcolsep{3pt}
		\begin{tabular}{|c|ccc:cc|cc:ccc|} 
			\hline
			& $e$ & $r$ & $r^2$ & $s$ & $sr$ & $r^5$ & $r^{10}$ & $sr^2$ & $sr^7$ & $sr^{12}$	\\ \hline
			$e$ &  &  &  &  &  & $r^8$ & $r^3$ & $sr^2$ & $sr^7$ & $sr^{12}$ \\
			$r$ &  &  &  &  &  & $r^9$ & $r^4$ & $sr$ & $sr^6$ & $sr^{11}$ \\
			$r^2$ &  &  &  &  &  & $r^{10}$ & $r^5$ & $s$ & $sr^5$ & $sr^4$ \\ \hdashline
			$s$ &  &  &  &  &  & $sr^8$ & $sr^3$ & $r^2$ & $r^7$ & $r^{12}$\\
			$sr$ &  &  &  &  &  & $sr^9$ & $sr^4$ & $r$ & $r^6$ & $r^{11}$ \\ \hline
			
			$r^5$ & $r^5$ & $r^4$ & $r^3$ & $sr^8$ & $sr^9$ &  &  &  &  &  \\
			$r^{10}$ & $r^{10}$ & $r^9$ & $r^8$ & $sr^3$ & $sr^4$ &  &  &  &  &  \\ \hdashline
			$sr^2$& $sr^2$ & $sr$ & $s$ & $r^{11}$ & $r^{12}$ &  &  &  &  & \\
			$sr^7$ & $sr^7$ & $sr^6$ & $sr^5$ & $r^6$ & $r^7$ &  &  &  &  &  \\
			$sr^{12}$ & $sr^{12}$ & $sr^{11}$ & $sr^{10}$ & $r$ & $r^2$ &  &  &  &  &  \\
			\hline
		\end{tabular}
	
		\caption{Table of $(26, 2, 5, 1)$-SEDF over $D_{26}$}
		\label{table2}
	
	\end{table}

\end{example}

\begin{theorem}\label{thm:dihedral}
Let $k>1$ be odd.  Let $\G=D_{k^2+1}$, the dihedral group of order $n=k^2+1$.  There exists a $(k^2+1,2,k,1)$-SEDF in $\G$. \\

Specifically, $\A=\{ A_1, A_2 \}$ is a $(k^2+1,2,k,1)$-SEDF where
\begin{itemize}
\item $A_1=\{e,s, r, sr, r^2, sr^2,\ldots, r^{\frac{k-1}{2}} \}$, \mbox{i.e. } $\{r^i: 0 \leq i \leq \frac{k-1}{2}\} \cup \{sr^j: 0 \leq j \leq \frac{k-3}{2} \}$.
\item $A_2=\{ sr^{\frac{k-1}{2}}, r^k, sr^{\frac{k-1}{2}} r^k, r^{2k}, sr^{\frac{k-1}{2}} r^{2k}, \ldots, r^{\frac{k(k-1)}{2}}, sr^{\frac{k-1}{2}}r^{\frac{k(k-1)}{2}} \}$,\\ \mbox{ i.e. } $\{r^{ik}: 1 \leq i \leq \frac{k-1}{2} \} \cup \{ sr^{jk+\frac{k-1}{2}}: 0 \leq j \leq \frac{k-1}{2} \}$.
\end{itemize}
\end{theorem}
\begin{proof}
 For disjoint subsets $X,Y$ of a group $\G$, let $\Delta(X,Y)=\{xy^{-1}:x \in X, y \in Y\}$.
It suffices to show that the multiset $\Delta(A_1,A_2)$ comprises each non-identity group element precisely once, since the multiset $\Delta(A_2,A_1)$ consists of the inverses of this multiset.

For $x,y \in \G$, $xy^{-1}$ equals:
\begin{itemize}
\item $r^{i-j}$ if $x=r^i$ and $y=r^j$;
\item $sr^{i-j}$ if $x=sr^i$ and $y=r^j$;
\item $r^{j-i}$ if $x=sr^i$ and $y=sr^j$;
\item $sr^{j-i}$ if $x=r^i$ and $y=sr^j$.
\end{itemize}
Note that $r^{(\frac{k^2+1}{2})}=e$.  

Obtaining the elements $r^i$ ($1 \leq i \leq \frac{k^2-1}{2}$) once each is equivalent to showing that in the additive group $(\mathbb{Z}_{\frac{k^2+1}{2}}, +)$, the multiset of differences between the pair of sets $\{0,1,\ldots,\frac{k-1}{2} \}$ and $\{k, 2k, \ldots, \frac{k(k-1)}{2} \}$ and the multiset of differences  between the pair of sets $\{\frac{k-1}{2}, \frac{3k-1}{2} ,\ldots,\frac{k^2-1}{2} \}$ and $\{0, 1, \ldots, \frac{k-3}{2} \}$ together comprise each non-zero element of $\mathbb{Z}_{\frac{k^2+1}{2}}$ once each.  It can be verified that this is the case: 
\begin{itemize}
\item $(jk+\frac{k-1}{2})-\{0,1,\ldots,\frac{k-3}{2} \}=\{jk+1, jk+2, \ldots, jk+\frac{k-1}{2}\}$ where  $0 \leq j \leq \frac{k-1}{2}$;
\item $ \{0, 1, \ldots, \frac{k-1}{2} \} - k(\frac{k-1}{2}-j) = \{jk+\frac{k+1}{2}, jk+\frac{k+3}{2}, \ldots, jk+k \}$ where $0 \leq j \leq \frac{k-3}{2}$.
\end{itemize}
For the second part, observe that $-k(\frac{k-1}{2}-j)=jk+(\frac{k+1}{2})$.

Obtaining the elements $sr^i$ ($0 \leq i \leq \frac{k^2-1}{2}$) once each corresponds to showing that the multiset of differences between the pair of sets $\{0, 1, \ldots, \frac{k-3}{2} \}$ and $\{k, 2k, \ldots, \frac{k(k-1)}{2} \}$ and the multiset of differences between the pair $\{\frac{k-1}{2}, \frac{3k-1}{2} ,\ldots,\frac{k^2-1}{2} \}$ and $\{0,1,\ldots,\frac{k-1}{2} \}$ together comprise all the elements of $\mathbb{Z}_{\frac{k^2+1}{2}}$ once each:
\begin{itemize}
\item $(jk+\frac{k-1}{2})-\{0,1,\ldots,\frac{k-1}{2} \}=\{jk, jk+1, \ldots, jk+\frac{k-1}{2}\}$ where  $0 \leq j \leq \frac{k-1}{2}$;
\item $ \{0, 1, \ldots, \frac{k-3}{2} \} - k(\frac{k-1}{2}-j)=\{jk+\frac{k+1}{2}, jk+\frac{k+3}{2}, \ldots, jk+(k-1) \}$ where $0 \leq j \leq \frac{k-3}{2}$.
\end{itemize}
The proof is demonstrated in Table \ref{table3}, which displays the multiset of external differences $\Delta(A_1,A_2)$.  The entry in the row labelled by $x$ and the column labelled by $y$ is $xy^{-1}$.  Each non-identity group element occurs precisely once in this table.

\begin{table}[h]
    \renewcommand{\arraystretch}{1.9}
		\centering
		\setlength\tabcolsep{3pt}
		\begin{tabular}{|c|ccccc:ccccc|} 
			\hline
			& $r^k$ & $r^{2k}$ & \ldots & $r^{k(\frac{k-3}{2})}$ & $r^{k(\frac{k-1}{2})}$ & $sr^{(\frac{k-1}{2})}$ & $sr^{(\frac{3k-1}{2})}$ & \ldots & $sr^{(\frac{(k-2)k-1}{2})}$ & $sr^{(\frac{k^2-1}{2})}$	\\ \hline
			$r^0=e$ & $r^{(\frac{k^2-2k+1}{2})}$ & $r^{(\frac{k^2-4k+1}{2})}$ & \ldots & $r^{(\frac{3k+1}{2})}$ & $r^{(\frac{k+1}{2})}$ & $sr^{(\frac{k-1}{2})}$ & $sr^{(\frac{3k-1}{2})}$ & \ldots & $sr^{(\frac{(k-2)k-1}{2})}$ & $sr^{(\frac{k^2-1}{2})}$ \\
			$r$ & $r^{(\frac{k^2-2k+3}{2})}$ & \ldots & \ldots & \ldots & $r^{(\frac{k+3}{2})}$ & $sr^{(\frac{k-3}{2})}$ & \ldots & \ldots & \ldots & $sr^{(\frac{k^2-3}{2})}$ \\
			$\vdots$ & \vdots & \ldots &  \ldots & \ldots & \vdots & \vdots & \ldots& \ldots & \ldots & \vdots \\
			$r^{(\frac{k-5}{2})}$ & $r^{k(\frac{k-1}{2})-2}$ & \ldots & \ldots  & $r^{2k-2}$ & $r^{k-2}$ & \vdots & \ldots & \ldots & \ldots & \vdots \\
			$r^{(\frac{k-3}{2})}$ & $r^{k(\frac{k-1}{2})-1}$ & $r^{k(\frac{k-3}{2})-1}$ &  \ldots & $r^{2k-1}$ & $r^{k-1}$ & $sr$ & $sr^{k+1}$ & \ldots & \ldots & $sr^{(\frac{k(k-1)}{2}+1)}$\\
			$r^{(\frac{k-1}{2})}$ & $r^{k(\frac{k-1}{2})}$ & $r^{k(\frac{k-3}{2})}$ & \ldots & $r^{2k}$& $r^k$ & $s$ & $sr^k$ & \ldots & \ldots & $sr^{(\frac{k(k-1)}{2})}$ \\ \hdashline
			
			$sr^0=s$ & $sr^{(\frac{k^2-2k+1}{2})}$ & $sr^{(\frac{k^2-4k+1}{2})}$ & \ldots & $sr^{(\frac{3k+1}{2})}$ & $sr^{(\frac{k+1}{2})}$  & $r^{(\frac{k-1}{2})}$ & $r^{(\frac{3k-1}{2})}$ & \ldots & $r^{(\frac{(k-2)k-1}{2})}$ & $r^{(\frac{k^2-1}{2})}$	\\
			$sr$ & $sr^{(\frac{k^2-2k+3}{2})}$ & \ldots & \ldots & \ldots & $sr^{(\frac{k+3}{2})}$ & $r^{(\frac{k-3}{2})}$ & $r^{(\frac{3k-3}{2})}$ & \ldots & \ldots & $ r^{(\frac{k^2-3}{2})}$ \\
			 $\vdots$ & \vdots & \ldots &  \ldots & \ldots & \vdots  & \vdots & \ldots &  \ldots & \ldots & \vdots\\
			$sr^{(\frac{k-5}{2})}$  & $sr^{k(\frac{k-1}{2})-2}$ & \ldots & \ldots  & $sr^{2k-2}$ & $sr^{k-2}$ &$ r^2 $ & $r^{k+2}$ & \ldots & \ldots & $r^{(\frac{k(k-1)}{2}+2)}$ \\
			$sr^{(\frac{k-3}{2})}$  & $sr^{k(\frac{k-1}{2})-1}$ & $sr^{k(\frac{k-3}{2})-1}$ &  \ldots & $sr^{2k-1}$ & $sr^{k-1}$ & $r$ & $r^{k+1}$ & \ldots & $r^{(\frac{k(k-3)}{2}+1)}$ & $r^{(\frac{k(k-1)}{2}+1)}$ \\
			\hline
		\end{tabular}
	
		\caption{Table showing $\Delta(A_1,A_2)$ for the $(k^2+1, 2, k, 1)$-SEDF over $D_{k^2+1}$ of Theorem \ref{thm:dihedral}}
		\label{table3}
	\end{table}

\end{proof}

We are now able to prove the following result.

\begin{corollary}
For every $k>2$, there exist at least two non-equivalent $(k^2+1, 2, k, 1)$-SEDFs. 
\end{corollary}
\begin{proof}
If $k$ is odd, we may take the construction of Proposition \ref{SEDFexistence}(1)  in $\mathbb{Z}_{k^2+1}$ and the construction of Theorem \ref{thm:dihedral} in $D_{k^2+1}$.  These are non-equivalent as the groups are not isomorphic. Otherwise, $k$ is even and $k \geq 4$, i.e. $k=2a$ for some $a \geq 2$, and so two non-equivalent constructions in $\mathbb{Z}_{k^2+1}$ are guaranteed by Theorem \ref{kcomposite}.
\end{proof}

We end this section with a first step towards addressing the question: are the only non-trivial SEDFs with $\lambda=1$ those with $m=2$?  

\begin{definition}
Let $\G$ be a group of order $n$.
\begin{itemize}
\item Let $\A$ be a set of $m \geq 2$ disjoint $k$-subsets of $\G$, say $A_1, \dots , A_m$.  We say that $\A$ is an $(n,m,k,\lambda)$-coEDF if the multiset
\begin{equation*}
N=\{ y^{-1}x : x \in A_i, y \in A_j, i \neq j \}
\end{equation*}
comprises $\lambda$ occurrences of each nonzero element of $\G$.
\item
Let $\A$ be a set of $m \geq 2$ disjoint $k$-subsets of $\G$, say $A_1, \dots , A_m$.  We say that $\A$ is an $(n,m,k,\lambda)$-coSEDF if, for every $i$, $1 \leq i \leq m$, the multiset
\begin{equation*}
N_i=\{ y^{-1}x: x \in A_i, y \in \cup_{j \neq i} A_j\}
\end{equation*}
comprises $\lambda$ occurrences of each nonzero element of $\G$.
\end{itemize}
\end{definition}

Clearly if $\G$ is an abelian group, coEDFs and coSEDFs coincide precisely with EDFs and SEDFs.   An $(n,m,k,\lambda)$-coSEDF must satisfy precisely the same parameter conditions as an $(n,m,k,\lambda)$-SEDF.

For a set $X$ in a group $\G$, we denote $X^{-1}=\{x^{-1}: x \in X \}$.  For a collection $\A$ of sets $\{A_1, \ldots, A_m \}$ in a group $\G$, we denote $\A^{-1}=\{ A_1^{-1}, \ldots, A_m^{-1}\}$.  The following result is immediate.

\begin{lemma}
$\A=\{A_1, \ldots, A_m \}$ is an $(n,m,k,\lambda)$-EDF (respectively, SEDF) if and only if $\A^{-1}=\{A_1^{-1}, \ldots, A_m^{-1} \}$ is an $(n,m,k,\lambda)$-coEDF (respectively, coSEDF).
\end{lemma}

\begin{example}\label{ex:dihedralex}
The $(10,2,3,1)$-SEDF $\A$ from Example \ref{dihedralex} in $D_{10}$, with sets $A_1=\{e,s,r\}$ and $A_2=\{sr,r^3,sr^4\}$ is also a $(10,2,3,1)$-coSEDF.  Unlike in the abelian case, the subtraction tables corresponding to the SEDF and coSEDF are different.  Here, $\A^{-1}$ comprises sets $A_1^{-1}=\{e,s,r^4\}$ and $A_2^{-1}=\{sr, r^2, sr^4\}$.  It can be checked directly that this is also a  $(10,2,3,1)$-SEDF in $D_{10}$.
\end{example}

\begin{proposition}
Let $\G$ be a group of order $n$.  Suppose that $\A=\{A_1, \ldots, A_m \}$ is an $(n,m,k,1)$-SEDF and an $(n,m,k,1)$-coSEDF (equivalently, that $\A$ and $\A^{-1}$ are both $(n,m,k,1)$-SEDFs).  Then either $m=2$ or $k=1$.
\end{proposition}
\begin{proof}
Suppose that $m>2$ and $k>1$.  Let $N^{\prime}_{ij}=\{ y^{-1}x: x \in A_i, y \in A_j \}$.  The multiset union $$\bigcup_{1 \leq i \leq m, \, 1 \leq j \leq n} N^{\prime}_{ij}$$
comprises every non-identity element of $\G$ precisely $m$ times.  The multiset union
$$\bigcup_{i \neq 1, \, j \neq 1} N^{\prime}_{ij}$$
comprises $m-2$ copies of each non-identity element of $\G$, since $\bigcup_{j \neq 1} N^{\prime}_{1j}(=N_1)$ comprises each non-identity element once, as does $\bigcup_{i \neq 1} N^{\prime}_{i1}$ (the set of inverses of $N_1$).

Let $x \neq y \in A_1$ (this is possible since $k>1$).  Set $\alpha=y^{-1}x$.  Since $\alpha$ is a non-identity element of $\G$ and $\A$ is an $(n,m,k,1)$-coSEDF in $\G$, there exist $u \in A_i$ and $v \in A_j$ for some $i, j \in \{2, \ldots, m\}$ with $i \neq j$ such that $v^{-1}u=\alpha$ (from above, this is possible since $m>2$).  Then $v^{-1}u=\alpha=y^{-1}x$, which rearranges to $v^{-1}=y^{-1}x u^{-1}$, i.e. $y v^{-1}=x u^{-1}$.  However, $\A$ is an $(n,m,k,1)$-SEDF, and this equality corresponds to two equal distinct external differences arising out of $A_1$ - a contradiction.
\end{proof}
This result generalizes the forward direction of Proposition \ref{PaStlambda=1}.  It indicates that, if a non-trivial non-abelian $(n,m,k,1)$-SEDF exists with $m>2$, it cannot be a $(n,m,k,1)$-coSEDF.

\section{Computational approach and results}

In this section, we discuss the computational approach which allowed us to find, and classify, all SEDFs in groups up to order 24.

\subsection{Algorithm description and analysis}

\begin{algorithm}

\caption{SEDF Finding}

\label{sedf:checkpart}\begin{algorithmic}[1]

\Procedure{CheckPartial}{$G,n,m,\lambda, L$}

\For{$i1 \in [1..m]$}

\State $Count \gets [0,0,\dots,0]$

\Comment{$Count$ is length $n$}

\For{$i2 \in [1..i1-1,i1+1..m]$}

\For{$j1 \in L[i1]$}

\For{$j2 \in L[i2]$}

\State$Count[j1*j2^{-1}]\ \mathrel{{+}{=}}\ 1$

\If{$Count[j1*j2^{-1}] > \lambda$}

\State\Return \textsc{False}

\EndIf





\EndFor

\EndFor

\EndFor

\EndFor

\State \Return \textsc{True}

\EndProcedure

\end{algorithmic}

\label{sedf:search}\begin{algorithmic}[1]

\Procedure{SEDFSearch}{$G,n,m,k,\lambda,L,p$}

\If{\textbf{not} \textsc{CheckPartial}$(G,n,m,\lambda,L)$}

\State\Return

\Comment{Too many occurrences of some difference}

\EndIf

\If{$m*k-(\sum s \in L. |s|) > n - p$}

\State\Return

\Comment{Not enough values remain to fill the SEDF}

\EndIf

\If{$\forall s \in L. |s| = k$}

\State\textbf{Output}$(L)$

\Comment{Found a valid, complete SEDF}

\State\Return

\EndIf

\State{\textsc{SEDFSearch}$(G,n,m,k,\lambda,L,p+1)$}

\Comment{Try SEDFs without G[p]}

\For{$i \in [1..m]$}

\If{$|L[i]| = k$}

\State\textbf{Continue}

\Comment{Continue on to next value for $i$}

\EndIf

\State{$L^{\prime} := L$}\Comment{Make copy of $L$}

\State{Add($L^{\prime}[i], G[p])$}

\State{\textsc{SEDFSearch}$(G,n,m,k,\lambda,L',p+1)$}

\If{$|L[i]| = 0$}

\State\Return \Comment{Only try adding a value to one empty set}

\EndIf

\EndFor

\EndProcedure

\end{algorithmic}

\label{sedf:search}\begin{algorithmic}[1]

\Procedure{SEDFSearchStart}{$G,n,m,k,\lambda$}

\Comment{$G$ is interpreted as a list, with identity first}

\State{$L \gets [\{G[1]\},\{\},\dots,\{\}]$}

\Comment{$L$ has length $m$}

\State{\textsc{SEDFSearch}{$(G,n,m,k,\lambda,L,2)$}}

\EndProcedure

\end{algorithmic}

\end{algorithm}

Computational search results were obtained via a constraint-style backtrack search.

Preprocessing of the group information was performed using GAP (\cite{GAP}), in particular its \textit{SmallGroups} library, while the search for SEDFs was implemented in Java, using a recursive depth-first search algorithm

Algorithm \ref{sedf:search} shows how we search for SEDFs. This algorithm will return all \((n,m,k,\lambda)\)-SEDFs in a group $G$. This algorithm assumes that $\lambda(n-1) = k^2(m-1)$ and that the elements of \(G\) are stored as a list, with the identity of \(G\) as the first element. As \(G\) is a list we will use the notation \(G[i]\) to access the \(i^{th}\) member of \(G\).  We order the elements of $G$ using their position in this list.

We build up the SEDF as a list of sets called \(L\). We build it one value at a time, by inserting elements of \(G\) into the members of \(L\). The algorithm begins with \textsc{SEDFSearchStart}. This function partially breaks the symmetry of left-shifting and right-shifting SEDFs, by requiring that the identity is always in the SEDF.

The function \textsc{SEDFSearch} builds the SEDF \(L\) incrementally. This function accepts a partially constructed SEDF \(L\), and a current position \(p\) in the group \(G\). This function will find all SEDFs which can be made by adding values from \(G[p]\) onwards to \(L\). \textsc{SEDFSearch} begins by trying to prove it is impossible to extend \(L\) to a valid SEDF. Firstly Line 2 calls \textsc{CheckPartial}, which calculates the differences between each pair of sets in \(L\) -- if any value occurs more than \(\lambda\) times we can reject \(L\). Line 4 then checks whether there are enough remaining values which could be inserted into \(L\) to complete the SEDF, and rejects if not.

In Line 6 we check if each element of \(L\) contains \(k\) values. In this case we have a complete candidate SEDF. As we know that $\lambda(n-1) = k^2(m-1)$ and in \textsc{CheckPartial} we calculated $k^2(m-1)$ differences and counted at most $ \lambda$ occurrences of each of the \(n-1\) non-identity elements of \(G\), then we must have counted exactly \(\lambda\) differences for each non-identity element, and therefore found a SEDF.

If we reach Line 8, we know we have a partially built SEDF, $L$. Firstly we try recursively building SEDFs which extend \(L\) and do not contain \(G[p]\). After that, we try building SEDFs where we add \(G[p]\) to each element of \(L\) which does not yet contain \(k\) elements.

One concern here is that we do not want to produce the same SEDF multiple times, where the only difference is that the elements of \(L\) are in a different order. Lemma \ref{lem:order} shows that by ensuring \(L\) is always ordered, we avoid producing the same SEDF multiple times. Line 16 gurantees that we only try adding a value to an empty element of \(L\) once per call to \textsc{SEDFSearch}.

\begin{lemma}\label{lem:order}

During the execution of \textsc{SEDFSearchStart}, the argument \(L\) of \\
\( \textsc{SEDFSearch}(G,n,m,k,\lambda,L,p) \) satisfies the conditions that
\begin{itemize}
\item there exists some \(i \in \{1, \ldots, m\}\) such that  \(\{ j: L[j] \neq \emptyset\} = \{1,\ldots, i\}\) and \(\{ j: L[j]=\emptyset\} = \{i+1,\ldots, m\}\), and
\item  \(\forall j \in \{1, \ldots, i-1\}. min(L[j]) < min(L[j+1])\)
\end{itemize}
 by ordering the elements of \(G\) (which we consider as a list) under the order \(x < y\) if \(x\) comes before \(y\).

\end{lemma}

\begin{proof}
We will proceed by induction. The first time \textsc{SEDFSearch} is called from \textsc{SEDFSearchStart}, only the first element of \(L\) is not the empty set, so the conditions follow trivially. The first time \textsc{SEDFSearch} recursively calls itself on Line 9 it passes the same value for \(L\). The subsequent recursive calls on Line 15 pass \(L\) with \(G[p]\) added to one member of \(L\).  If \(G[p]\) is inserted into a non-empty member of \(L\), then since this value is larger than any value already in \(L\), it will not affect the minimum, and the conditions hold. Otherwise, \(G[p]\) is added to an empty member of \(L\). This will be done for the first empty element of \(L\), where \(G[p]\) will be larger than any other value in any member of \(L\), and so the conditions are again satisfied. After this, Line 16 will cause \textsc{SEDFSearch} to immediately return and not insert \(G[p]\) into any later elements of \(L\).

\end{proof}

After \textsc{SEDFSearchState} has produced a list of SEDFs, we produce a final list of non-equivalent SEDFs using the Images package \cite{JeJoPfWa} in GAP.

\subsection{Summary of results}

Table \ref{table4} lists all admissible parameters sets for groups of order at most $24$.  

	\begin{table}[h]
	\begin{center}
	\begin{tabular}{c|c|cc|cc|cc|ccccc|cccc|cc}
	$n$ & 5 & 9 & & 10 & & 13 & & 17 & & & & & 19 & & & &  21 &  \\  \hline
	$m$ & 2 & 2& 3 & 2& 3 & 2& 4& 2& 2& 3& 4& 5& 2& 3& 3& 5& 2& 6  \\ \hline
	$k$ & 2 & 4& 2 & 3& 3 & 6& 2& 4& 8& 4& 4& 2& 6& 3& 6& 3& 10& 2  \\ \hline
	$\lambda$ & 1 & 2& 1 & 1& 2 & 3& 1& 1& 4& 2& 3& 1& 2& 1& 4& 2& 5& 1  \\
	\end{tabular}
	\caption{All admissible SEDF parameters for $n \leq 24$}
	\label{table4}
	\end{center}
	\end{table}
	
For abelian groups, the following parameter sets are ruled-out by results from the literature:
\begin{itemize}
\item Proposition \ref{nonexistence}(1): $(9,3,2,1)$, $(10,3,3,2)$, $(13,4,2,1)$, $(17,3,4,2)$, $(17,4,4,3)$, $(19,3,3,1)$, $(19,3,6,4)$;
\item Proposition \ref{nonexistence}(2): $(17,5,2,1)$, $(19,5,3,2)$;
\item Proposition \ref{PaStlambda=1}: $(21,6,2,1)$;
\item Theorem 4.3 of \cite{JeLi}: $(19,2,6,2)$, $(21,2,10,5)$.
\end{itemize}


Table \ref{table5} summarizes all non-equivalent abelian SEDFS found by computer search in groups of order up to $24$.   All found SEDFs can be classified in terms of constructions in the literature; the cases are as follows:
\begin{itemize}
\item[(a)] Proposition \ref{SEDFexistence} (1): construction $\{0,1,\ldots,k-1\}$, $\{k,2k, \ldots, k^2 \}$ given in \cite{PaSt};
\item[(b)] Proposition \ref{SEDFexistence} (2): construction corresponding to Paley difference sets, given in \cite{HuPa}: for prime power $n=q$, take the squares/non-squares of the multiplicative group of $\mathbb{F}_q$;
\item[(c)] Proposition \ref{SEDFexistence} (3): construction using cyclotomic classes in the multiplicative group of $\mathbb{F}_q$, given in \cite{BaJiWeZh}.
\item[(d)] Theorem \ref{newZconstruction}: construction $\{0,1, \ldots, a-1\} \cup \{2a, 2a+1, \ldots 3a-1\}$, $\cup_{i=1}^a \{ (4i-1)a, 4ia\}$.
\end{itemize}

\begin{table}[ht]
    \renewcommand{\arraystretch}{1.7}
\begin{center}
		\begin{tabular}{c|c|c|c|c}
			Parameters &{ Abelian group} & No of non-equiv. SEDFs& Example & Case \\ \hline
		$(5, 2 ,2 , 1  )$ & $\mathbb{Z}_5$   & 1 & $\{0,1\}, \{2,4\}$ & (a), (b), (d)  \\ \hline
                 \multirow{2}{*}{$(9,2,4,2)$}& $\mathbb{Z}_9$   & 0 & - & -  \\ 
		& $\mathbb{Z}_ 3 \times \mathbb{Z}_3$ & 1  & \makecell {$\{(1,0),(0,1,)(2,0),(0,2)\},$ \\ $ \{(1,1),(1,2),(2,1),(2,2)\}$} & (b)     \\ \hline
		$(10, 2 ,3 , 1  )$ & $\mathbb{Z}_{10} $  & 1 & $ \{0,1,2\}, \{3,6,9\}$ &  (a)  \\ \hline
		
		$(13, 2 , 6  ,3)$ & $\mathbb{Z}_{13}$   & 1 &  \makecell {$\{1,3,4,9,10,12\},$ \\ $ \{ 2,5,6,7,8,11 \}$} & (b)     \\ \hline
		
		$(17, 2 , 4  ,1)$ & $\mathbb{Z}_{17}$   & 2 & \makecell {$\{0,1,2,3\},\{4,8,12,16\}$ \\ $ \{1,4,13,16\},\{2,8,9,15 \}$}  & \makecell {(a) \\ (c), (d)}    \\ \hline
				
		$(17, 2 , 8  ,4)$ & $\mathbb{Z}_{17}$   & 1 & \makecell {$\{1,2,4,8,9,13,15,16\},$ \\ $ \{ 3,5,6,7,10,11,12,14 \}$} & (b)   \\ 
		\end{tabular}
		\caption{Non-equivalent SEDFs in abelian groups of order up to $24$}
		\label{table5}
\end{center}
\end{table}
There is no $(9,2,4,2)$-SEDF in the cyclic group of order $9$; this is ruled-out by Theorem 4.2 of \cite{JeLi}.  The two non-equivalent $(17,2,4,1)$-SEDFs are those guaranteed by Theorem \ref{kcomposite}; here the SEDF from the cyclotomic construction of Proposition \ref{SEDFexistence} (3) is equivalent to that from Theorem \ref{newZconstruction} (see Example \ref{equiv:ex}).

For non-abelian groups of order $n \leq 24$, we cannot rule-out any of the parameter sets from Table \ref{table4} using results from the existing literature.  However, there are no non-abelian groups of orders 5, 9, 13, 17, or 19. Therefore non-trivial SEDFs in groups of order up to $24$ are possible only for two orders, $10$ and $21$.  In each case there is one non-abelian group: the dihedral group $D_{10}$ of order 10 and  the semi-direct product ${\mathbb{Z}_7 \rtimes \mathbb{Z}_3}$ of order 21.  The non-trivial parameter sets for these are $(10,2,3,1)$, $(21,2,10,5)$ and $(21,6,2,1)$.  The results are shown in Table \ref{table6}.  

\begin{table}[ht]
    \renewcommand{\arraystretch}{1.4}
\begin{center}
		\begin{tabular}{c|c|c|c|c}
			Parameters &{ Non-abelian group} & No of non-equiv. SEDFs& Example & Case \\ \hline
		$(10, 2 ,3 , 1  )$ & $D_{10}$   & 1 & $\{e,s,r \}, \{sr, r^3, sr^4\}$ & Theorem \ref{thm:dihedral} \\ \hline
                 $(21, 2,10,5 )$ & ${\mathbb{Z}_7 \rtimes \mathbb{Z}_3}$  & 0 & - & - \\ \hline
		
		$(21, 6 ,2,1 )$ & ${\mathbb{Z}_7 \rtimes \mathbb{Z}_3}$  & 0 &  - & - \\ 
		
		\end{tabular}
		\caption{Non-equivalent SEDFs in non-abelian groups of order up to $24$}
		\label{table6}
\end{center}
\end{table}

\section{Concluding remarks and open problems}

In this paper, we have begun to explore the questions: for given parameters, which finite groups (abelian or non-abelian) contain SEDFs with these parameters? What are the possible structures of these SEDFs?  How many non-equivalent SEDFs exist with the same parameters?   The concept of equivalence also raises new questions about SEDF results in the literature - for example, how many non-equivalent $(243,11,22,20)$-SEDFs exist?

The case when $\lambda=1$ possesses a richer structural landscape than previously realised, with new SEDFs found in both abelian and non-abelian groups, and a wide range of abelian constructions obtainable recursively.  For no other fixed value of $\lambda$ are explicit families of SEDFs known; all other SEDF constructions have $\lambda$ as a function of $n$.  It is an open question whether  families can be found for fixed values of $\lambda>1$.

Further constructions for non-abelian SEDFs would be of interest.  Does there exist a SEDF with parameters which cannot be realised in an abelian group but can be obtained in a non-abelian setting? This is of particular interest as the known range of possible parameters for abelian SEDFs becomes increasingly constrained (eg. in  \cite{LeLiPr} it is conjectured that the known $(243,11,22,20)$-SEDF is the only abelian SEDF with $m>2$).  Do there exist non-abelian SEDFs with $\lambda=1$ and $m>2$?

Finally, we may ask whether certain phenomena observable in the computational results for SEDFs in groups of order up to $24$, are representative of a more general situation.  Is it the case that in a group of order $p^2$ for $p$ an odd prime, no SEDF with $m=2$ can exist in the cyclic group of order $p^2$?    In \cite{BaJiWeZh}, it is proved that there does not exist a $(p^2,m,k,\lambda)$-SEDF in an abelian group $\G$ with $p$ prime, $k>1$, $m>2$ and $\G$ cyclic; however we know of no such result in the literature for $m=2$. Does the finite field construction using squares and non-squares yield a unique (up to equivalence) $(p^2,2, \frac{p-1}{2}, \frac{p-1}{4})$-SEDF?

\subsection*{Acknowledgements} Thanks to Maura Paterson for her helpful comments on the paper, and to Gemma Crowe and Ailsa Robertson for related discussions.  The second author is supported by a Royal Society University Research Fellowship.

\end{document}